\documentclass[11pt,a4paper]{amsart}
\usepackage{a4wide}
\usepackage{amsfonts,amsthm,amsmath,amssymb,amscd}
\usepackage{graphicx,color}
\usepackage[shortlabels]{enumitem}

\theoremstyle{plain}
\newtheorem{lem}{Lemma}%[section]
\newtheorem{prop}[lem]{Proposition}
\newtheorem{thm}[lem]{Theorem}
\newtheorem{cor}[lem]{Corollary}
\newtheorem*{thmA}{Main Theorem}

\theoremstyle{definition}

\newtheorem*{defn*}{Definition}

\newtheorem*{ex*}{Example}
\newtheorem{rem}[lem]{Remark}
\newtheorem*{rem*}{Remark}

\theoremstyle{remark}

\DeclareMathOperator{\dist}{dist}

\DeclareMathOperator{\ext}{ext}
\DeclareMathOperator{\inter}{int}
\DeclareMathOperator{\wind}{wind}
\DeclareMathOperator{\Id}{Id}
\DeclareMathOperator{\Fix}{Fix}

\newcommand{\C}{\mathbb C}
\newcommand{\D}{\mathbb D}

\newcommand{\chat}{\widehat{\C}}
\newcommand{\clC}{\widehat \C}

\newcommand{\N}{\mathbb N}

\renewcommand{\SS}{\mathcal S}
\newcommand{\bd}{\partial}

\newcommand{\Omegat}{\widetilde{\Omega}}

\newcommand{\fat}{\mathcal{F}}
\newcommand{\jul}{\mathcal{J}}

\begin{document}

\title[Connectivity of Julia sets 
of Newton maps: A unified approach]{Connectivity of Julia sets 
of Newton maps:\\ A unified approach}

\date{\today}

\author{Krzysztof Bara\'nski}
\address{Institute of Mathematics, University of Warsaw,
ul.~Banacha~2, 02-097 Warszawa, Poland}
\email{baranski@mimuw.edu.pl}

\author{N\'uria Fagella}
\address{Departament de Matem\`atica Aplicada i An\`alisi,
Universitat de Barcelona, 08007 Barce\-lona, Spain}
\email{fagella@maia.ub.es}

\author{Xavier Jarque}
\address{Departament de Matem\`atica Aplicada i An\`alisi,
Universitat de Barcelona, 08007 Barce\-lona, Spain}
\email{xavier.jarque@ub.edu}

\author{Bogus{\l}awa Karpi\'nska}
\address{Faculty of Mathematics and Information Science, Warsaw
University of Technology, ul.~Ko\-szy\-ko\-wa~75, 00-662 Warszawa, Poland}
\email{bkarpin@mini.pw.edu.pl}

\thanks{Supported by the Polish NCN grant decision DEC-2012/06/M/ST1/00168. The second and third authors were partially supported by the Catalan grant 2009SGR-792, and by the Spanish grant MTM2011-26995-C02-02.}

\bibliographystyle{amsalpha}

\begin{abstract}

%In this paper we give a unified proof of the fact that the Julia set of Newton's method applied to a holomorphic function of the complex plane (a polynomial of degree large than $1$ or an entire transcendental function) is connected. The result was recently completed by the authors' previous work, but the proof is spread among many papers, considering separately a number of particular cases for rational and transcendental maps, and using a variety of techniques. In this note we present a unified, simpler  and reasonably self-contained proof which works for all situations alike.

In this paper we give a unified proof of the fact that the Julia set of Newton's method applied to a holomorphic function of the complex plane (a polynomial of degree large than $1$ or an entire transcendental function) is connected. The result was recently completed by the authors' previous work, as a consequence of a more general theorem whose proof  spreads among many papers, which consider separately a number of particular cases for rational and transcendental maps, and use a variety of techniques. In this note we present a unified, direct and reasonably self-contained proof which works for all situations alike.
\end{abstract}

\maketitle

\section{Introduction}\label{sec:intro}

Newton's method is one of the oldest and best known root-finding algorithms. It is also the motivation which inspired  the modern approach to  holomorphic dynamics, when the local study turned out to be insufficient for a good understanding of the method applied to complex polynomials.
%holomorphic functions in the Riemann sphere or the complex plane. 

The global dynamics of Newton's method applied to complex quadratic polynomials is always conjugate to the dynamics of $z\mapsto z^2$, as already noticed in the early works of E.~Schr\"oder and A.~Cayley \cite{cayley1,cayley2,cayley3,schr1,schr2}. They also  observed that this trivial situation is no longer true when Newton's method is applied to higher degree polynomials, where the boundaries between different basins of attraction of attracting fixed points (known nowadays as the Julia set) have, in general, rich and intricate topology. 

A good understanding of the topology of the Julia set of Newton's method, applied to either polynomials or entire transcendental functions, is  interesting not only from the point of view of holomorphic dynamics but it also has interesting numerical applications \cite{hubbardschleicher}. One of the questions which has attracted much attention over many years is whether the stable components of the method, including for example the basins of attraction of the attracting fixed points, are simply connected. We know now that the answer is affirmative as a corollary of a more general theorem, whose proof spreads over the papers \cite{shishikura, FJT1,FJT2,berter} and \cite{bfjk}. The  proofs used various topological and analytical techniques, including quasiconformal geometry. 

Our goal in this paper is to give a direct and unified proof of the connectivity of the Julia set of Newton's method, or equivalently,  of the simple connectivity of each of its stable components. Our proof is inspired by the new approach  introduced in  \cite{bfjk}, which included  the development and applications of  fixed point theorems. These techniques are now expanded and refined so that they fullfil the new goals.  We now proceed to describe our objectives in more detail. 

Let $g:\C  \to \C$ be a polynomial of degree $d\geq 2$ or an entire transcendental map, i.e. a holomorphic map on $\C$ with an essential singularity at infinity. Its \emph{Newton's method} (called also the \emph{Newton map} corresponding to $g$) is defined as
$$
N=N_g:=\Id - \frac{g}{g^{\prime}}.
$$
It is well known that the finite fixed points of $N$ are, exactly, the zeroes of $g$. Moreover, all of them are attracting (the derivative of $N$ has modulus smaller than $1$ at these points). In fact, if the corresponding root of $g$ is simple, then the fixed point of $N$ is superattracting (the derivative of $N$ vanishes). 
 
If $g$ is a polynomial of degree $d\geq 2$, then $N$ is a rational map, and hence it is holomorphic on the Riemann sphere $\clC$. It is easy to check that in this case the point at infinity is a repelling fixed point of $N$. If $N$ is the Newton map of an entire transcendental function $g$, then $N$ is meromorphic transcendental, with infinity being an essential singularity, except for the case $g(z)=P(z)\exp(Q(z))$ with $P$ and $Q$ polynomials, when $N$ is rational. (In this very special case, the point at infinity is a parabolic fixed point of $N$ with derivative $1$.) In both cases, all finite fixed points of $N$ are attracting.

\begin{figure}[htp!] \label{fig:newton}
\centering
\includegraphics[width=0.48\textwidth]{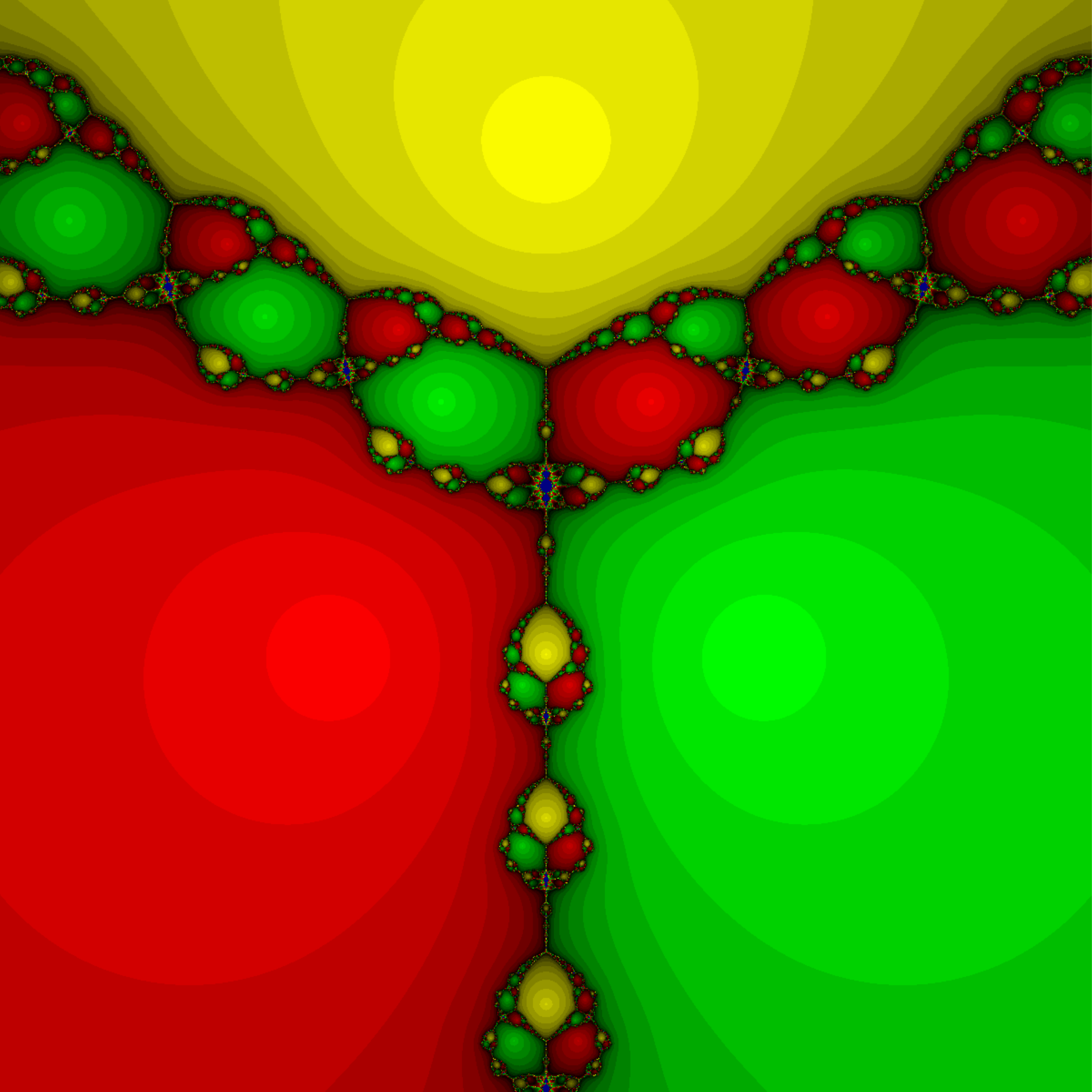} \hfill
\includegraphics[width=0.48\textwidth]{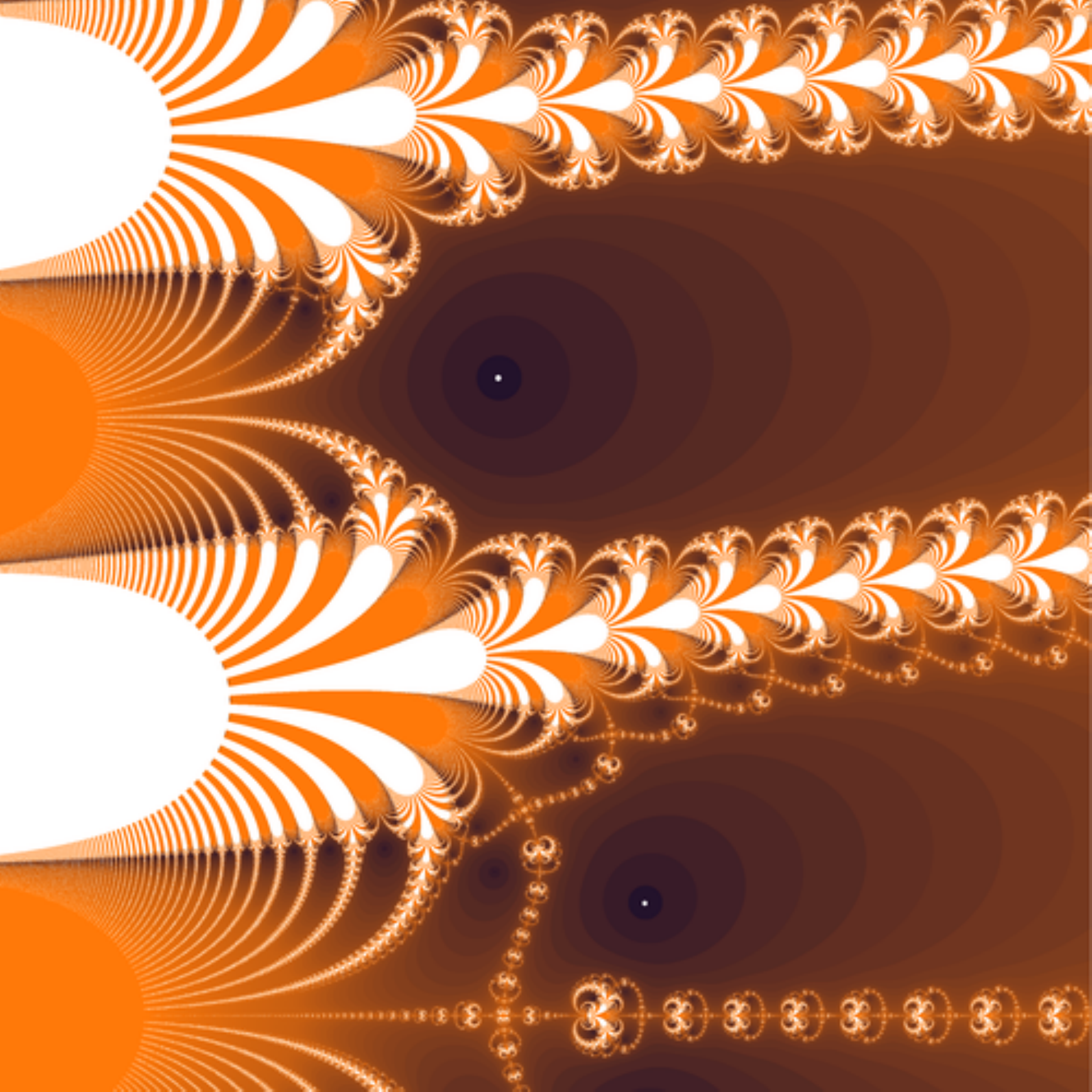} 

\caption{Dynamical planes of Newton's methods for a polynomial (left) and entire transcendental map (right).}
\end{figure}

Let $f:\C \to \chat$ be a meromorphic map, rational or transcendental, as for example the Newton's map $N$ of a polynomial of degree larger than $1$ or a transcendental entire map. 
We consider the dynamical system given by the iterates of $f$, which induces  a dynamical partition of the Riemann sphere into two completely invariant sets: the {\em Fatou set} $\fat(f)$, which is the set of points $z\in\chat$, where the family of iterates $\{f^k\}_{k\geq 0}$ is defined and normal in some neighborhood of $z$, and its complement, the {\em Julia set} $\jul(f) = \chat \setminus \fat(f)$. The Fatou set is open and consists of points with, in some sense, stable dynamics, while the Julia set is closed and its points exhibit chaotic behavior. For general background on the dynamics of rational and meromorphic maps we refer to, for example, \cite{bergweiler,carlesongamelin} or \cite{milnor}. Note that for all Newton's maps $N$, the point at infinity is contained in the Julia set of $N$.

Connected components of the Fatou set, known as {\em Fatou components}, are mapped by $f$ among themselves. A Fatou component $U$ is {\em periodic} of (minimal) period $p \ge 1$, or {\em $p$-periodic}, if $f^p(U) \subset U$. For $p=1$ such a component is named \emph{invariant}. A component which is not eventually periodic (under the iteration of $f$) is called {\em wandering} (these do not exist in the rational case \cite{sullivan}). There is a complete classification of periodic Fatou components: such a component can either be a {\em rotation domain} (a Siegel disc or a Herman ring), the {\em basin of attraction} of an attracting or parabolic periodic point or a {\em Baker domain}, although  this last possibility occurs only  if the map is transcendental. Recall that a $p$-periodic Fatou component $U\subset \C$ is a Baker domain, if $f^{pk}$ on $U$ tend to a point $\zeta$ in the boundary of $U$ as $k \to \infty$, and $f^j(\zeta)$ is not defined for some $j \in\{0,\ldots p-1\}$. This implies the existence of an unbounded Fatou component $U'$ in the same cycle, such that 
$f^{pk}\to \infty$ on $U'$. 

As already mentioned, the question of the connectivity of the Julia set of meromorphic Newton's maps has been widely considered in the literature.  Note that, since the Julia set is compact in $\chat$,  its connectivity is equivalent to the simple connectivity of all Fatou components. 
The first  results are due to F.~Przytycki \cite{przytycki} and Tan Lei \cite{tanlei} on rational Newton's maps. A complete answer for Newton maps for polynomials was given by M.~Shishikura \cite{shishikura} via a more general theorem. More precisely, by means of quasiconformal surgery, he proved
 that  every rational map with less than two weakly repelling fixed points (i.e.~fixed points $z_0$  such that $|f'(z_0)|>1$ or $f'(z_0)=1$), as it is the case of rational Newton maps, has a connected Julia set. 

The extension of this remarkable result to transcendental Newton maps turned out to be not easy. Shishikura's techniques which were based on the pullback of invariant absorbing sets under $N$,   encountered technical difficulties in the transcendental setting due to the presence of the essential singularity at infinity. Nevertheless, with some extra tools, the strategy worked for all Fatou components except for Baker domains \cite{berter,FJT1,FJT2}. The case of Baker domains required a  new approach which was recently developed by the authors in \cite{bfjk}.  Prior to the actual proof,  the existence of  nice enough absorbing domains was shown for this type of Fatou components. Once this was settled, an alternative strategy to Shishikura's pullback construction was presented, providing the existence of weakly repelling fixed points and therefore completing the proof  by contradiction. 

In this paper we want to apply these new ideas to the whole range of possible Fatou components, not only for Baker domains. Our aim is not to reprove the general result of \cite{shishikura, FJT1,FJT2,berter,bfjk}, but to restrict to Newton's method and give a unified proof of the connectivity of its Julia set,  using a common and simpler strategy for both, the rational and the transcendental case. Therefore our goal is to prove the following theorem. 

\begin{thmA}\label{theorem:A}
Let $g$ be a holomorphic function on the complex plane of degree larger than $1$ $($polynomial or entire transcendental$)$ and let $N_g$ be its Newton's method. Then the Julia set of $N_g$ is a connected subset of $\clC$ or, equivalently, every Fatou component of $N_g$ is a simply connected subset of $\mathbb C$.
\end{thmA}

We shall use  auxilliary results which provide the existence of fixed points assuming certain relative positions of sets and their respective images. Some of these results (or slight variations therein) were developed in \cite{bfjk} but others are new,  based on the argument principle and homotopies. In all cases we argue by contradiction, using on the one hand that all finite fixed points of Newton's method are attracting and on the other hand that their basins of attraction are unbounded. 

The paper is organized as follows. In Section~\ref{sec:back} we state the tools and preliminary results used in the proof. Some of these results were previously developed in \cite{bfjk}, nevertheless we include the proofs of the main lemmas  to make the paper self-contained. The proof of the Main Theorem is contained in Section~\ref{section:proofA}.

\section*{Acknowledgments}
We are grateful Mitsuhiro Shishikura for his motivating encouragement to write this unfied proof. We wish to thank the Institut de Matem\`atica de la Universitat de Barcelona (IMUB), the Mathematical Institute of the Polish Academy of Sciences, the University of Warsaw and the Technical University of Warsaw for their hospitality. We also thank Jordi Canela and Antonio Garijo for useful discussions. 

%%%%%%%%%%%%%%%%%%%%%%%%%
\section{Background and preliminary results}\label{sec:back}

\subsection{Fatou components of meromorphic maps} 
From the beginning of the 20th century, it is well known that basins of attraction of attracting or parabolic cycles possess simply connected \emph{absorbing sets}. Indeed, if $U$ is the basin of a (super)attracting cycle, there exists a neighborhood of the periodic orbit which is invariant under the map and which eventually captures the orbit of every point in $U$. A set with similar properties can also be constructed for the basin of a parabolic cycle (see Remark \ref{rem:scabsorbing}).  

In the realm of transcendental dynamics there appear periodic Fatou components of an additional type, namely Baker domains. These are sometimes called parabolic domains at infinity because they reflect the dynamics of parabolic basins with the parabolic cycle containing infinity. Nevertheless, the fact that the essential singularity is part of the (virtual) cycle introduces significant differences in the study of the dynamics. Although the Fatou theory for parabolic cycles does not apply here, in some cases one can achieve a reasonable understanding of the dynamics near infinity (see e.g.~\cite{barfag,faghen,rippon,bakerexamples}). In particular, the existence of absorbing sets (not necessarily simply connected) inside Baker domains was recently established in \cite{bfjk}. We state this result below, in an appropriate form to be applied in the proof of the Main Theorem. 

%\begin{thm}[\bf Existence of absorbing regions for holomorphic self-maps of hyperbolic domains] \label{theorem:absorbing}
%Let $U$ be a hyperbolic domain in $\C$ and let $F: U \to U$ be a holomorphic
%map, such that $F^n \to \infty$ as $n \to \infty$. Then for every point $z \in U$ and every sequence of positive numbers $r_n$, $n \geq 0$ with $\lim_{n\to\infty} r_n = \infty$,
%there exists a domain $W\subset U$, such that:
%\begin{itemize}
%\item[$(a)$]$W \subset \bigcup_{n=0}^\infty \DD_U(F^n(z), r_n)$,
%\item[$(b)$] $\overline{W} \subset U$,
%\item[$(c)$] $F^n(\overline{W}) =
%\overline{F^n(W)} \subset F^{n-1}(W)$ for every $n\geq 1$,
%\item[$(d)$] $\bigcap_{n=0}^\infty F^n(\overline{W})
%= \emptyset$,
%\item[$(e)$] $W$ is absorbing in $U$ for $F$.
%\end{itemize}
%Moreover, $F$ is locally univalent on $W$. 
%\end{thm}
%

\begin{thm}[\bf Existence of absorbing regions in Baker domains] 
\label{absorbing}
Let $f: \C \to \clC$ be a transcendental meromorphic map and let $U$ be a periodic Baker domain of period $p$ such that $f^{pn} \to \infty$ as $n \to \infty$. Set $F:=f^p$. Then there exists a domain $W \subset U$ with the following properties:
\begin{enumerate}[\rm (a)]
%\item[$(a)$]$W \subset \bigcup_{n=0}^\infty \DD_U(F^n(z), r_n)$,
\item $\overline{W} \subset U$,
\item $F^n(\overline{W}) =
\overline{F^n(W)} \subset F^{n-1}(W)$ for every $n\geq 1$,
\item $\bigcap_{n=0}^\infty F^n(\overline{W})
= \emptyset$,
\item $W$ is absorbing in $U$ for $F$, i.e.~for every compact set $K\in U$, there exists $n_0\in\N$ such that $F^n(K)\subset W$ for all $n\geq n_0$.
\end{enumerate}
Moreover, $F$ is locally univalent on $W$.
\end{thm}

\begin{rem}
\label{rem:scabsorbing}
If $U$ is the basin of a (super)attracting $p$-periodic point $\zeta$, then $F = f^p$ is conformally conjugate to $z\mapsto F'(\zeta) z$ (if $F'(\zeta)\neq 0$) or $z\mapsto z^k$ for some integer $k\ge 2$ (if $F'(\zeta)= 0$) near $z=0$. In this case, if we take $W$ to be the preimage of a small disc centered at $z=0$ under the conjugating map, then $W$ is a simply connected absorbing domain for $F$ and $\bigcap_{n\geq 0} F^n(\overline{W}) = \{\zeta\}$. Similarly if $U$ is a basin of a parabolic $p$-periodic point, it has a simply connected absorbing domain in  an attracting petal in $U$. 
\end{rem}

The following result appears as part of \cite{mayer}. We include here a direct proof for completeness using ideas from \cite{przytycki} (see also \cite{fagstandard,rempe,deniz}).

\begin{prop}[\bf Unboundedness of Newton's basins] 
\label{prop:unbounded_attracting_basins}
Let $N$ be a meromorphic Newton's map (rational or transcendental) and let $U$ be the immediate basin of attraction of an attracting fixed point. Then $U$ is unbounded.
\end{prop}

\begin{proof}
Assume that $U$ is an immediate bounded basin of attraction of an attracting fixed point  $\zeta$. Since $U$ is bounded, it contains finitely many critical points and all of them are attracted to $\zeta$. Hence we can choose two distinct points $z_0, z_1 \in U\setminus \{\zeta\}$  such that $N(z_1)=z_0$, and they can be joined by a curve $\gamma_0\subset U\setminus\overline{\bigcup_{n\geq 1}N^n({\rm Crit} \cap U)}$, where ${\rm Crit}$ denotes the set of critical points of $N$. 

Denote by $h$ the local branch of $N^{-1}$ mapping  $z_0$ to $z_1$. This branch can be extended along $\gamma_0$ unless $\gamma_0$ contains an asymtpotic value whose asymptotic path is contained in $U$, contradicting the boundness of $U$. Repeating the argument, we can define inductively
$\gamma_n=h(\gamma_{n-1})$ where now $h$ denotes  the extension of the initial branch along the curve $\bigcup_{j=0}^{n-1}\gamma_j$. Set $\gamma=\bigcup_{n=0}^\infty\gamma_n.$

Observe that  there exists a neighborhood $V$ of $\gamma_0$ such that the distortion of $h^n$ is bounded on $V$, independently of $n$. This implies that the diameter of $\gamma_n$ tends to 0. Indeed, otherwise there exists a subsequence $n_j$ and a  nonempty  open set $V' \subset \cap_{n=1}^\infty h^{n_j}(V)$ containing a limit point of $\gamma$, so the family $N^{n_j}$ is normal on $V'$. But this is not possible since any limit point of $\gamma$ is in the Julia set.

Hence $|z_n-z_{n+1}|=|z_n - N(z_n)|  \to 0$ and therefore there is  a finite fixed point in $\partial U$, which is a contradiction.

\end{proof}

%%%%%%%%%%%%%%%%%%
\subsection{Images of curves and existence of fixed points}
The notation and the results in this section will be used repeatedly in the proof of the Main Theorem. 

For a compact set $X \subset \C$ we denote by $\ext(X)$ the connected component of $\clC \setminus X$ containing infinity.  We set $K(X) = \clC \setminus \ext(X)$ and notice that $K(X)$ is closed and bounded. 
If $f$ is a holomorphic map with no poles in a neighborhood of $K(X)$, then by the Maximum Principle, $f(K(X))=K(f(X))$. For a Jordan curve $\gamma \subset \C$ we denote by $\inter(\gamma)$ the bounded component of $\C \setminus \gamma$.

The first result in this sequel establishes the existence of poles in some bounded component of the complement of a multiply connected Fatou component. This will be the starting point in most of our future arguments.  
 
\begin{lem}[\bf Poles in loops] \label{lem:poles-in-holes} 
Let $f:\C \to \clC$ be a meromorphic transcendental map or a rational map for which infinity belongs to the Julia set. Let $\gamma \subset \C$ be a
closed curve in a Fatou component $U$ of $f$, such that $K(\gamma) \cap \jul(f) \ne
\emptyset$. Then there exists $n \geq 0$, such that $K(f^n(\gamma))$ contains
a pole of $f$. Consequently, if $U$ is multiply connected then there exists a bounded component of $\clC \setminus f^n(U)$, which contains a pole.
\end{lem}

\begin{proof}
If $f$ is transcendental, it is well known that prepoles are dense in the Julia set \cite{bergweiler}. If $f$ is rational, preimages of any given point in the Julia set, in particular infinity, are dense in the Julia set.  

Let $\gamma \subset \C$ be a
closed curve in a Fatou component $U$ of $f$, such that $K(\gamma) \cap \jul(f) \ne
\emptyset$. By the observation above,  $K(\gamma)$ contains a prepole  of order, say, $n\geq 0$, where $n$ is the smallest with this property. By the maximum principle, $f^j(K(\gamma))= K(f^j(\gamma))$ for all $j\leq n$ and therefore $K(f^n(\gamma))$ contains a pole of $f$.  Since $f^n(\gamma) \subset f^n(U)$, it follows that if $U$ is multiply connected, the pole belongs to a bounded component of $\clC\setminus f^n(U)$.
\end{proof}

The remaining statements ensure the existence of weakly repelling fixed points under certain hypotheses. Recall that a fixed point $z_0$ of a holomorphic map $f$ is weakly repelling, if $|f'(z_0)| > 1$ or $f'(z_0) = 1$. The main lemmas  rely heavily on the following two theorems due to Buff. 

\begin{thm}[\bf Rational-like maps {\cite[Theorem 2]{buff}}] \label{rat-like}
Let $D$ and $D'$ be domains in $\C$ with finite Euler
characteristic, such that $\overline{D'}
\subset D$ and let $f:D' \to D$ be a proper holomorphic map. Then $f$ has a weakly repelling fixed point in $D'$.
\end{thm}

The following is an improved corollary of \cite[Theorem 3]{buff}. 
\begin{thm}[\bf Rational-like maps with boundary contact {\cite[Corollary 2.12]{bfjk}}]\label{cor:boundarycontact}
Let $D$ be a simply connected domain in $\clC$ with locally connected
boundary and $D^{\prime}\subset D$ a domain in $\clC$ with finite Euler
characteristic. Let $f$ be a continuous map on the closure of $D'$ in $\clC$,
meromorphic in $D'$, such that $f:D^{\prime}\to D$ is proper. 
If $\deg f > 1$ and $f$ has no fixed points in $\partial D \cap \partial D'$, or $\deg f = 1$ and $D \ne D^{\prime}$,
then $f$ has a weakly repelling fixed point in $D^{\prime}$. 
\end{thm}

We shall also use the following topological result.
\begin{thm}[\bf Torhorst Theorem {\cite[pp. 106, Theorem 2.2]{whyburn}}] \label{tor}
Let $X$ be a  locally connected continuum in $\chat$. Then the boundary of every component of $\chat\setminus X$ is a locally connected continuum.
\end{thm}

The folllowing are the  main results which will be used in our proofs. All of them (or slight modifications therein), except Proposition \ref{prop:index}, were proven in \cite{bfjk}. The latter is new and its proof is contained in Subsection \ref{thenewproof}. 

\begin{lem}[\bf Boundary maps out] \label{mapout}
Let $\Omega \subset \C$ be a bounded domain with finite Euler
characteristic and let $f$ be a meromorphic map 
in a neighborhood of $\overline{\Omega}$. Assume that there exists a component
$D$ of \ $\clC \setminus
f(\bd\Omega)$, such that:
\begin{itemize}
\item[$(a)$] $\overline{\Omega} \subset D$,
\item[$(b)$] there exists $z_0 \in \Omega$ such that $f(z_0) \in D$. 
\end{itemize}
Then $f$ has a weakly repelling fixed point in $\Omega$. Moreover, if
additionally $\Omega$ is simply connected with locally
connected boundary, then the assumption~$(a)$ can be replaced by: 
\begin{itemize}
\item[$(a')$] $\Omega \subsetneq D$ and $f$ has no fixed points in
$\bd \Omega \cap f(\bd\Omega)$ 
\end{itemize}
or by
\begin{itemize}
\item[$(a'')$]  $\Omega=D$, $f$ has no fixed points in $\bd \Omega$ and $f(\Omega)\neq \Omega$.
\end{itemize}
\end{lem}

\begin{proof}
By the assumption (b), there exists a component $D^{\prime}$ of $f^{-1}(D)$
containing $z_0$. Observe that 
\[
D^{\prime} \subset \Omega.
\]
To see this, suppose that $D^{\prime}$ is not contained in $\Omega$. 
Then there exists $z\in D^{\prime}\cap \bd\Omega$. Consequently, $f(z)\in 
D \cap f(\bd\Omega)$. This is a contradiction since, by definition, $D \cap 
f(\bd\Omega) =\emptyset$.

As a consequence, $D^{\prime}$ is bounded. Moreover, since $\Omega$ has finite
Euler characteristic, $\bd\Omega$ (and hence $f(\bd\Omega)$ and $\bd D$) has a
finite number of components, so $D$ has finite Euler characteristic. One can check that $D'$ has finite Euler characteristic and the restriction $f:D^{\prime} \to D$ is
proper. Moreover, the assumption~(a) implies $\overline{D'}
\subset D$. Hence (possibly after a change of coordinates in
$\clC$ by a M\"obius transformation), $f:D^{\prime} \to D$ is a 
rational-like map, so $f$ has a weakly repelling fixed point 
in $D^{\prime} \subset \Omega$ by Theorem \ref{rat-like}.

Now, assume that $\Omega$ is simply connected with locally
connected boundary, and the assumption~(a) is replaced by~(a$'$). Then $\bd\Omega$ (and hence $f(\bd\Omega)$) is a 
locally connected continuum in $\clC$. Moreover we also have that $D$ is
simply connected and, by  Torhorst Theorem (if $X$ is a locally connected 
continuum in $\C$, then the boundary of every component of $\C\setminus X$ is a locally connected continuum) \cite[pp. 106, Theorem 2.2]{whyburn}, has locally connected boundary. Moreover, since $D'\subset
\Omega \subset D$ and the boundary of $D$ is contained in $f(\bd \Omega)$, the
intersection of the boundaries of $D$ and $D'$ is either empty or is contained
in $\bd \Omega \cap f(\bd \Omega)$. 
This together with the condition~(a$'$) implies that the restriction $f:D^{\prime} \to D$ satisfies the assumptions of Theorem \ref{cor:boundarycontact}, providing the existence of a weakly repelling fixed point. 

Finally suppose that (a$''$) is satisfied instead of (a$'$), so that $\partial \Omega \subset \partial f(\Omega)$.  Again, let $D'\subset \Omega$ be the connected component of $f^{-1}(D)$ containing $z_0$. By assumption, there exist points in $\Omega$ which do not map into $\Omega$ hence $D'\subsetneq D$. Since $f:D'\to D$ is proper, it has no fixed points in $\bd D' \cap \partial D$ and $D'\neq D$ we are again under the assumptions of Theorem \ref{cor:boundarycontact} which ends the proof.
\end{proof}

Lemma~\ref{mapout} implies the following two corollaries.

\begin{cor}[\bf Continuum surrounds a pole and maps out] \label{cor:mapout1} 
Let $X \subset \C$ 
be a continuum and let  $f$ be a meromorphic map in a neighborhood of $K(X)$. Suppose that:
\begin{itemize}
\item[$(a)$] $f$ has no poles in $X$, 
\item[$(b)$] $K(X)$ contains a pole of $f$,
\item[$(c)$] $K(X) \subset \ext(f(X))$.
\end{itemize}
Then $f$ has a weakly repelling fixed point in the interior of $K(X)$.
\end{cor}

\begin{proof} Let $p \in K(X)$ be a pole of $f$. Observe that by the assumption~(a), the set $f(X)$ (and hence $K(f(X))$) is a continuum in
$\C$. Moreover, (a) implies
\[
p \in \Omega \subset \overline\Omega \subset K(X)
\]
for a bounded simply connected component $\Omega$ of $\clC \setminus X$. 
We have $\bd\Omega \subset X$, which gives $f(\bd\Omega) \subset f(X)$, so by the assumption~(c),
\[
K(X) \subset \ext(f(\bd\Omega)),
\]
which implies $\overline\Omega \subset\ext(f(\bd\Omega))$.

Let $D = \ext(f(\bd\Omega))$. We have $\overline\Omega \subset D$, 
$p \in \Omega$ and $f(p) = \infty \in D$. Hence, the assumptions of
Lemma~\ref{mapout} are satisfied for
$\Omega, D, p$, so $f$ has a weakly
repelling fixed point in $\Omega$, which is a subset of the interior of
$K(X)$.
\end{proof}

\begin{cor}[\bf Continuum maps out twice] \label{cor:mapout2} 
Let $X \subset \C$ be a continuum and let
$f$ be a meromorphic map in a neighbourhood of $X \cup K(f(X))$.
Suppose that:
\begin{itemize}
\item[$(a)$] $f$ has no poles in $X$, 
\item[$(b)$] $X \subset K(f(X))$,
\item[$(c)$] $f^2(X) \subset \ext(f(X))$.
\end{itemize}
Then $f$ has a weakly repelling fixed point in the interior of $K(f(X))$.
\end{cor}

\begin{proof} By the assumption (a), the set $f(X)$ (and hence $K(f(X))$) is a continuum in
$\C$ and $f^2(X)$ is a continuum in $\clC$. Moreover, $X \cap f(X) = \emptyset$
(otherwise $f(X) \cap f^2(X) \ne \emptyset$, which contradicts the assumption~(c)). Hence, by (b),
\[
X \subset \Omega \subset \overline\Omega \subset K(f(X))
\]
for some bounded simply connected component $\Omega$ of $\clC \setminus f(X)$. 
We have $\bd\Omega \subset f(X)$, so $f(\bd\Omega) \subset f^2(X)$ and by the assumption~(c),
\[
K(f(X)) \subset \clC \setminus f^2(X) \subset \clC \setminus f(\bd\Omega),
\]
which gives $K(f(X)) \subset D$ for some component $D$ of $\clC
\setminus f(\bd\Omega)$. Consequently, $\overline\Omega \subset K(f(X))\subset D$. Moreover, for any $z_0 \in X$ we have $z_0 \in \Omega$ and $f(z_0) \in f(X)
\subset D$. Hence, the assumptions of Lemma~\ref{mapout} are satisfied for
$\Omega, D$ and $z_0$, so $f$ has a weakly
repelling fixed point in $\Omega$, which is contained in the interior of
$K(f(X))$.
\end{proof}

The next proposition is new and it will be key in our arguments. Recall that the multiplicity of  a point $z_0$ fixed by a holomorphic map $f$ is the order of $z_0$ as a zero of $f(z)-z$.

\begin{prop}\label{prop:index} Let $\Omega \subset \C$ be a simply connected bounded domain and let $f$ be a meromorphic map in a neighborhood of $\overline{\Omega}$, such that $f(\partial \Omega) \subset \Omega$. Then $\Omega$ contains exactly $m+1$ fixed points of $f$, counted with multiplicities, where $m$ is the number of poles of $f$ contained in $\Omega$, counted with multiplicities. 
\end{prop}

\begin{rem*}
Notice that the number of fixed points of $f$ in $\Omega$ counted with multiplicity is the sum of the {\em Lefschetz indices} of the fixed points in $\Omega$. For similar results on the sum of the Lefschetz indices for holomorphic maps and relation to the Lefschetz Fixed Point Theorem in special situations, refer to \cite{milgol,johannes}. 
\end{rem*}

Observe that Proposition~\ref{prop:index}, as opposed to the lemmas above, does not give any information about the nature of the fixed points: these might be attracting, repelling or indifferent.  

%%%%%%%%%%%%%%%%%%
\subsubsection{Winding numbers: Proof of Proposition \rm \ref{prop:index}}
\label{thenewproof} \ \\
Given a closed oriented curve $\gamma:[0,1] \to \C$  and a point $P$ not in $\gamma$, we denote by $\wind(\sigma,P)$ the {\em  winding number (or index)} of $\gamma$ with respect to the point $P$, i.e.~the number of turns that $\gamma$ makes around $P$.  We will use the symbol $\gamma$ for both the curve and its image in the plane, $\gamma([0,1])$. The following is a simple application of  the Argument Principle  to the map $f(z)-z$.

\begin{lem}[\bf Argument Principle] \label{argpple}
Let $\Omega\subset \C$ be a domain bounded by  a Jordan curve $\gamma:[0,1]\to \C$ and let $f$ be a meromorphic map in a neighborhood of $\overline{\Omega}$  such that $f(z)\neq z,\infty$ for all points $z\in \gamma$. Set $\sigma(t):=f(\gamma(t))$. Let $\Fix (f)$ be the set of fixed points of $f$ and $P(f)$ be the set of poles of $f$. Then 
\[ 
\wind(\sigma(t)-\gamma(t),0) = \# (\Fix (f) \cap \Omega) - \# (P(f) \cap \Omega) 
\]
counted with multiplicities. 
\end{lem}

Hence, Lemma \ref{argpple} gives us the number of fixed points of a map $f$  inside a Jordan domain (counted with multiplicity),  if we know the number of poles, and provided we are able to compute the winding number of the curve $f(\gamma(t))-\gamma(t)$ with respect to the origin. But in many occasions this is not  an obvious computation to make. The following lemma  simplifies this counting  in the case that both curves do not intersect.  See  \cite[Lemma 4.6]{benfag} for a more general statement.

\begin{lem}[\bf Computing winding numbers] \label{homolemma}
Let   $\gamma, \sigma:[0,1] \to \C$ be  two disjoint closed curves and let $P\in \gamma$ and $Q\in \sigma$ be arbitrary points.  Then 
 \begin{equation}\label{lem:wind}
 \wind(\sigma(t)-\gamma(t),0)=\wind(\gamma,Q) + \wind(\sigma,P).
 \end{equation}
\end{lem}
\begin{proof}
Note that the right hand side in (\ref{lem:wind}) is independent of the chosen points $P\in\gamma$ and $Q\in \sigma$. Indeed, wind$(\gamma,.)$   is constant in every connected component of  $\C\setminus \gamma$, and by hypothesis, $\sigma(t)$ is contained in the same component for all $t$.  Similarly, the second term is also independent of $P$.

We now show the equality. Suppose first that $\gamma$ belongs to a bounded component of $\C\setminus \sigma$. Then, the left side of the equation is invariant under small perturbations of $\gamma$ and therefore is invariant under homotopies of $\gamma$ in $\C\setminus \sigma$. By contracting $\gamma$ to the constant curve $P$ we have that 
\[ 
\wind(\sigma(t)-\gamma(t),0)=\wind(\sigma -P,0) =  \wind(\sigma,P),
\]
and the equality (\ref{lem:wind}) follows since $\wind(P,Q)=0$. In the symmetric case, when $\sigma$ belongs to a bounded component of $\C\setminus \gamma$, we may contract $\sigma$ to the constant curve $Q$ and proceed equivalently.
\end{proof}

With these two tools we are now ready to proof  Proposition~\ref{prop:index}. 
\begin{proof}[Proof of Proposition \rm\ref{prop:index}]
It follows from the assumptions that  $\partial \Omega$ contains neither poles nor fixed points of $f$. Since fixed points and poles are isolated in $\C$, this is still true for a sufficiently small neighborhood of $\partial \Omega$,  say $V:=\{z\in\C \mid \dist(z,\partial \Omega)<\varepsilon\}$.  Decreasing $\varepsilon$ if necessary, we may assume, by continuity of $f$, that $f(\Omega\cap V) \subset \Omega\setminus V$. 
  
Let $\varphi:\D \to \Omega$ be a Riemann map 
and set $\gamma:=\varphi(\{u \in \mathbb D: |u|=1-\delta\})$ for a small $\delta > 0$, with the canonical parametrization. Since $\varphi$ is univalent, $\gamma$ is a Jordan curve and, for sufficienlty small $\delta$, it is contained in $V\cap \Omega$.
It follows that $f(\gamma)\subset \inter(\gamma)$. Moreover, $\inter(\gamma)$ contains $m$ poles, and exactly as many fixed points as $\Omega$ does. 
 
Set $\sigma(t):=f\left(\gamma(t)\right)$ and let $P=\sigma(0)=\sigma(1)$.
Then  it is clear that $\gamma\cap \sigma=\emptyset$ and hence we are under the hypothesis of  Lemma \ref{homolemma}. Notice  that $\wind(\gamma,P)=1$ because $\gamma$ is Jordan curve and $P\in \inter(\gamma)$. Likewise, $\wind(\sigma, z_0) =0$ for all $z_0\in\gamma$, given that $\sigma\subset\inter(\gamma)$.
 Thus
  \[
 \wind(\sigma(t)-\gamma(t),0)=\wind(\gamma,P) + \wind(\sigma,z_0)=1,
 \]
which together with Lemma \ref{argpple} yields 
% \[
% \wind(\sigma(t)-\gamma(t),0)\ =\ 1\ = \ \# (\Fix (f) \cap \inter(\gamma)) \ - \ \# (P(f) \cap \inter(\gamma)), 
% \]
 \[
 \# (\Fix (f) \cap \inter(\gamma)) \ - \ \# (P(f) \cap \inter(\gamma)) =   \wind(\sigma(t)-\gamma(t),0)\ =\ 1.
 \]
Therefore
\[
\# (\Fix (f) \cap \Omega) \ = \ \# (\Fix (f) \cap \inter(\gamma)) \ = \ m+1.
\]
  \end{proof}
%%%%%%%%%%%%%%%%%%%%%%%%%%%%%%%%%%
\section{Proof of the Main Theorem}\label{section:proofA}

In this section $N$ denotes a Newton map, that is, the Newton's method applied to a polynomial or to an entire transcendental function.

We shall prove the Main Theorem by showing that every possible Fatou component $U$ of a $N$ is simply connected. It is important to keep in mind that Newton maps have no finite weakly repelling fixed points since all their finite fixed points are attracting. 

We divide the proof into two cases. 
\begin{enumerate}[\rm (a)]
\item $U$ is an invariant Fatou component (Theorem \ref{invbasins}).
\item $U$ is a (pre)periodic Fatou component of minimal period $p>1$ or a wandering domain (Theorem \ref{thm:second}).
\end{enumerate}

%%%%%%
\subsection{Invariant Fatou components} 

Let $N$ be a meromorphic Newton's map. According to the Fatou Classification Theorem \cite[Theorem 6]{bergweiler}  if $U$ is a (forward) invariant Fatou domain then $U$ is the immediate (super)attracting basin of an attracting or parabolic fixed point , an invariant Herman ring or an invariant Baker domain. Since we are dealing with a Newton's method, the parabolic case is not possible unless $N$ is of the special type when $\infty$ is a parabolic fixed point with derivative one, in which case $U$ is its invariant parabolic basin.  Our goal in this section is to prove the following theorem. 

\begin{thm}[\bf Forward invariant Fatou components] \label{invbasins}
Let $N$ be a meromorphic Newton's map and let $U$ be a forward invariant Fatou component of $N$. Then $U$ is simply connected.
\end{thm}

We start by showing that invariant Herman rings cannot exist for $N$. A different proof can be found in \cite{johannes}.
%an invariant Baker domains with multiply connected absorbing sets (see Theorem \ref{theorem:absorbing} and see \cite{bfjk} for a detailed discussion of absorbing sets for Baker's domains). 

\begin{prop}
\label{nohermans}
A Newton map $N$ has no invariant Herman rings.
\end{prop}
\begin{proof}
Suppose $N$ has an invariant Herman ring $U$. Then $U$ is conformally equivalent to an annulus, foliated by simple closed curves which are invariant under $N$, and on which the dynamics is conjugate to an irrational rigid rotation. Choose $\gamma$ to be one such curve, and let $\Omega$ be the domain bounded by $\gamma$. Observe that  $f(\Omega \,\cap \, U)=\Omega \, \cap \, U$, but there must be points in $\Omega$ which are mapped outside $\Omega$ or, otherwise, $\{N^n\}$ would form a normal family in $\Omega$ by Montel's Theorem, contradicting that $\Omega\cap J(N)\neq \emptyset$.   Therefore  we are under the hypotheses of  Lemma \ref{mapout} (a''). Indeed,  $\Omega$   is simply connected with locally connected boundary and $N(\bd \Omega)=\bd \Omega$ so we may choose $D=\Omega$.   Moreover, $N$ has no fixed points in $\bd \Omega$ because $N|_{\partial \Omega}$ is conjugate to an irrational rotation. Therefore, by Lemma \ref{mapout},  we conclude that $N$ has a weakly repelling fixed point in $\Omega$, a contradiction. 
\end{proof}

Our next step is to prove that Baker domains for Newton's method always admit simply connected absorbing sets. 

\begin{prop}[\bf Simply connected absorbing sets]
\label{scabsorbing}
Let $N$ be  a Newton map and let $U$ be  an invariant  Baker domain of $N$. Then $U$ has a simply connected absorbing domain. 
\end{prop}

\begin{proof}
By Theorem \ref{absorbing}, we know that $U$ has an absorbing set $W$ such that $\overline{W}\subset U$ and $\overline{N(W)} \subset W$. Assume  that $W$ cannot be chosen to be simply connected.  Then, there exists a closed curve $\gamma \subset W$ such that $K(\gamma) \cap \jul(N) \neq \emptyset$. By Lemma  \ref{lem:poles-in-holes}, there exists $n\geq 0$, such that $K(N^n(\gamma))$ contains a pole $p$ of $N$. Let $\Omega$ be the connected component of $\C \setminus \overline{W}$ containing $p$. Since $W$ is connected,  $\Omega$ is simply connected. 

Given that $\overline{N(W)} \subset W$, we know that $N(\partial \Omega)\subset  \C\setminus \overline{\Omega}$, in particular $\partial \Omega\cap N(\partial \Omega)=\emptyset$. Now we  have  two possibilities: either $\overline{\Omega} \subset \ext(N(\partial \Omega))$ or $\overline{\Omega} \subset K(N(\partial \Omega))$. In the first case, we use Corollary \ref{cor:mapout1} with $X=\partial \Omega$ to obtain a weakly repelling fixed point in $\Omega$, a contradiction. So we may assume that 
\[
\overline{\Omega} \subset K(N(\partial \Omega)).
\]

Let
\[
\SS= \{s \ge 0: p \text{ is contained in a bounded component of } \C \setminus \overline{N^s(W)}\}.
\]
Note that $0 \in \SS$, so $\sup \SS$ is well defined.
We consider two further subcases.

\subsubsection*{Case $($\rm{i}$)$: $\sup\SS= S< \infty$}

Then $p$ is contained in a bounded
component $\Omega'$ of $\C \setminus \overline{N^S(W)}$ but is not contained in any bounded component of
$\C \setminus \overline{N^{S+1}(W)}$. Moreover, by Theorem \ref{absorbing} we have
\[
N(\bd \Omega') \subset N(\overline{N^S(W)}) = 
N^{S+1}(\overline{W}) \subset
N^S(W) \subset \C \setminus \overline{\Omega'}.
\]
This implies that $\overline{\Omega'}\subset \ext(N(\bd\Omega'))$. 
Consequently, the assumptions of Corollary~\ref{cor:mapout1} 
are satisfied for $X = \bd \Omega'$, and so $N$ has a weakly repelling fixed point in $\Omega'$, which is impossible. 

\subsubsection*{Case $($\rm{ii}$)$: $\sup\SS = \infty$}

Fix some point $z_0\in\C$, which is not a pole of
$N$. By assumption and Theorem \ref{absorbing}, for sufficiently large $n$ there
exists a bounded component $\Omega''$ of $\C \setminus \overline{N^n(W)}$
containing $p, z_0, N(z_0)$, such that
\[
N(\bd \Omega'') \subset N(\overline{N^n(W)}) = N^{n+1}(\overline{W})
\subset N^n(W) \subset \C \setminus \overline{\Omega''}. 
\]
Hence,
\[
\overline{\Omega''} \subset D,
\]
where $D$ is a component of $\clC \setminus N(\bd \Omega'')$. We have $z_0,
N(z_0) \in \Omega'' \subset D$. Hence, $\Omega'', D, z_0$ satisfy the
assumptions of Lemma~\ref{mapout} (a), from which we conclude that $N$ has a weakly repelling fixed point in $\Omega''$, a contradiction.
\end{proof}

We are now ready to prove Theorem \ref{invbasins}.
\begin{proof}[Proof of Theorem \ref{invbasins}]
In view of Propositions \ref{nohermans} and \ref{scabsorbing}, we may assume that the  invariant Fatou component $U$ is a (super)attracting immediate basin, or a Baker domain with a simply connected absorbing set or a parabolic immediate basin of the point at infinity (only for rational Newton maps). In each of these  three cases, there is a simply connected absorbing set included in $U$  (see  Remark \ref{rem:scabsorbing}), and hence the iterates of any closed curve are eventually contractible.
 
We assume that $U$ is multiply connected. Under this assumption,  Lemma \ref{lem:poles-in-holes} provides a simple closed curve   $\gamma' \subset U$ so that $\inter(\gamma')$ contains a pole of $N$, say $p$. Consider the set 
%\[
%\Gamma':=
%\begin{cases} 
%\overline{\bigcup_{n\geq 0} N^n(\gamma')} & \text{\ if $U$ is an attracting basin} \\
%\bigcup_{n\geq 0} N^n(\gamma') & \text{\ otherwise}.
%%\end{cases}
%\]

$$
\Gamma':=
\bigcup_{n\geq 0} N^n(\gamma').
$$

Clearly $\Gamma'$ is  forward invariant, i.e., $N\left(\Gamma'\right)\subset \Gamma'$. 
%If $U$ is an attracting basin, $\Gamma'$ is compact in $\C$. Otherwise, $\Gamma'$ is closed in $\C$ and unbounded.
 Notice that $p\notin \Gamma'$.

Iterates of $\gamma'$ must be eventually contractible. Hence there exists $n_0>0$ such that $p\in K(N^{n_0}(\Gamma'))$, but $p\notin K(N^{n}(\Gamma'))$ for all $n>n_0$. Set 
\[
\Gamma:=N^{n_0}(\Gamma').
\] 
Note that $\Gamma$ is a closed set in $\C$ except in some special cases where $\overline{\Gamma}\setminus \Gamma$ (where the closure is taken in $\C$) may consist of an attracting fixed point.

Let $\Omega'$  be the connected component of $\C\setminus \overline{\Gamma}$ containing $p$   and let
\[
 \Omega= \bigcup \{K(\sigma)\mid \sigma \text{ is a closed curve in $\Omega'$} \}.
 \]
By definition, $\Omega$ is a bounded simply connected domain in $\C$  containing $p$ and such that 
\[
\partial \Omega \subset \partial \Omega' \subset \overline{\Gamma} \subset U.
\]

Since $N(\Gamma)\subset\Gamma$,  one of the following must be satisfied:
\[
N(\partial \Omega) \cap \Omega =\emptyset \quad {\rm or} \quad N(\partial \Omega) \subset \Omega.
\]
Hence, using that no iterate of $\Gamma$ can surround the pole $p$, we have to consider the following two cases:
\begin{equation} \label{twocases}
\Omega \subset \ext(N(\bd \Omega)) \quad {\rm or} \quad N(\partial \Omega) \subset \Omega.
\end{equation}

\subsubsection*{Case $(i)$: $\Omega \subset \ext(N(\bd \Omega))$ } \ \\
Let us first assume that  $\overline{\Omega} \subset \ext(N(\bd \Omega)) $.  Then  we are under the hypotheses of Corollary \ref{cor:mapout1} with $X=\bd \Omega$, which provides a weakly repelling fixed point of $N$ in $\Omega$, which is impossible. 

If  $\Omega \subset \ext(N(\bd \Omega))$ but $\bd \Omega$ intersects its image,  we must again distinguish between two possibilities. Suppose that  $\partial \Omega$ contains no fixed point of $N$ (see Figure \ref{fig:sketch1} (a)). Then $\partial \Omega$  is a finite union of arcs contained in a finite number of iterates of the original curve $\gamma'$, and hence it is locally connected. It follows that  we are under the hypothesis of Lemma \ref{mapout} (a') with $D=\ext (N(\bd \Omega))$ and $z_0=p$ and, hence, there is a weakly repelling fixed point in $\Omega$, again a contradiction. 

We are left with the situation where $\Omega \subset \ext(N(\bd \Omega))$ and $\partial \Omega$ contains a fixed point $\zeta$ of $N$ (see Figure \ref{fig:sketch1} (b)).  This implies that $U$ is an attracting basin of an attracting fixed point  $\zeta$. In this case we must proceed in a slightly different way. 

We first observe that $\zeta$ is the only fixed point in $\bd \Omega$, given that $\partial \Omega\subset U$. Let $\Delta$ be a small topological disk containing  $\zeta$ such that $N(\overline{\Delta}) \subset \Delta$ (it exists because  $\zeta$ is attracting). Let $\Omegat:=\Omega \setminus \Delta$. By construction, $\partial \Omegat$ is  connected and therefore $\Omegat$ is simply connected. Moreover, $\partial \Omegat$ is locally connected since iterates of $\Gamma$ must eventually enter $\Delta$, hence, as above, $\partial \Omegat$ consists of finitely many arcs. Finally, since $\zeta\notin\partial \Omegat$, we are under the hypothesis of Lemma \ref{mapout} (a'), which again gives a contradiction.

\subsubsection*{Case $(ii)$: $N(\partial \Omega) \subset \Omega$} (See Figure \ref{fig:sketch1} (c)) \ \\
In this case the assumptions of Proposition \ref{prop:index} are satisfied. Indeed, $\bd \Omega$  is locally connected since, as in the previous cases, it consists of finitely many arcs (notice that $\bd \Omega$ cannot contain any fixed point in this case, and hence it is disjoint from the appropriate absorbing set). We  conclude from Proposition \ref{prop:index} that $\Omega$ contains at least two fixed points, since $\Omega$ contains at least one pole. One of them may possibly be the attracting fixed point in $U$ (if $U$ happened to be an attracting basin), but the second one  belongs to a different attracting basin, say  $U'$.  However $U' \subset \Omega$ and hence it is bounded, a contradiction with Proposition \ref{prop:unbounded_attracting_basins}.
  
We conclude that $U$ is simply connected and the proof is finished.
 \end{proof}

\enlargethispage{0.5cm}

\begin{figure}[htp!] 
\centering
\includegraphics[width=0.32\textwidth]{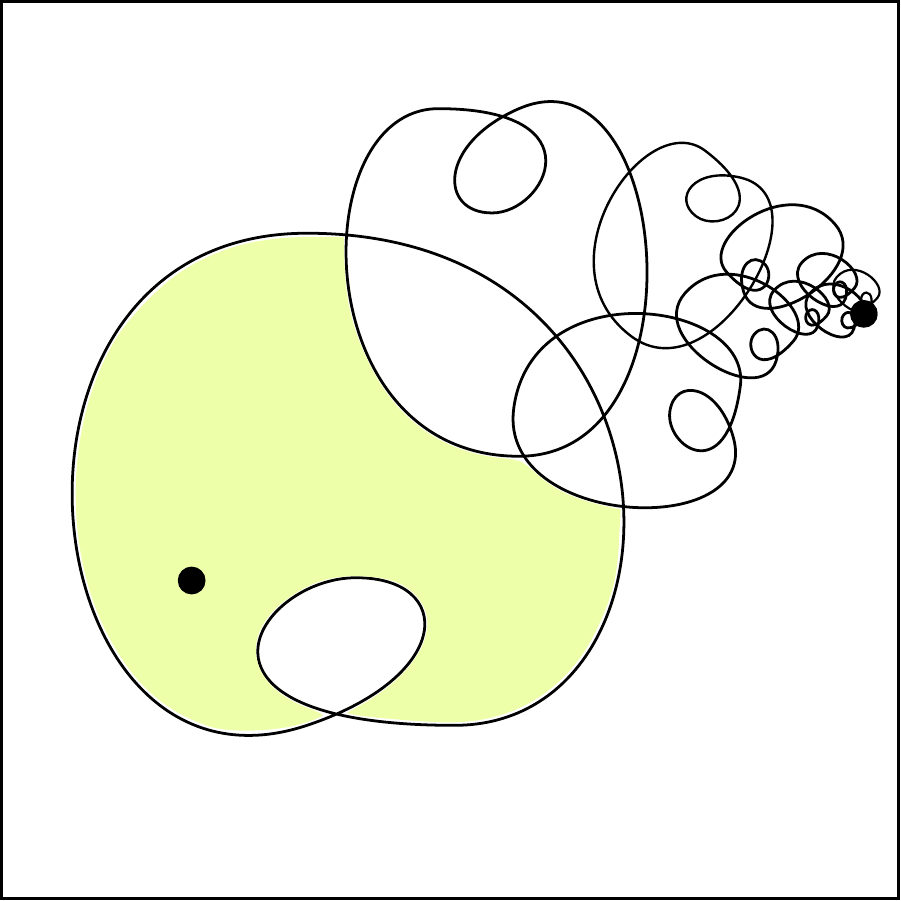} \hfill
\includegraphics[width=0.32\textwidth]{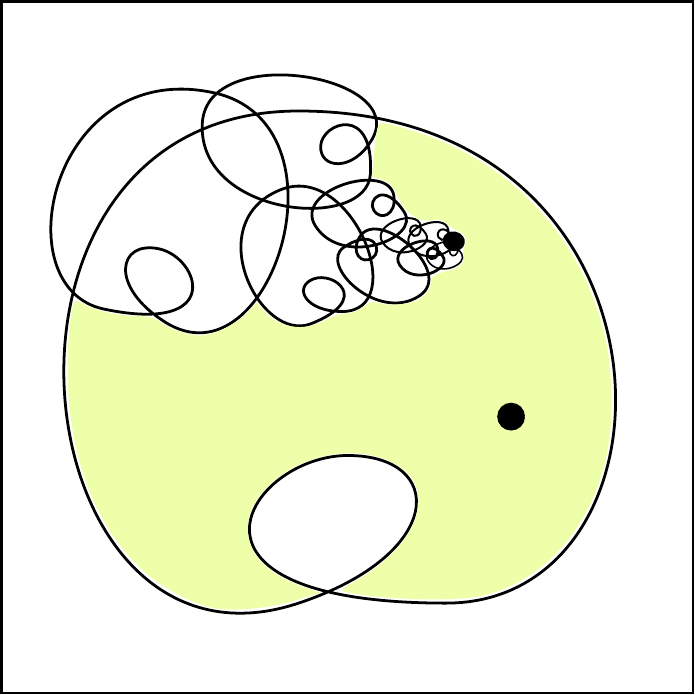} \hfill
\includegraphics[width=0.32\textwidth]{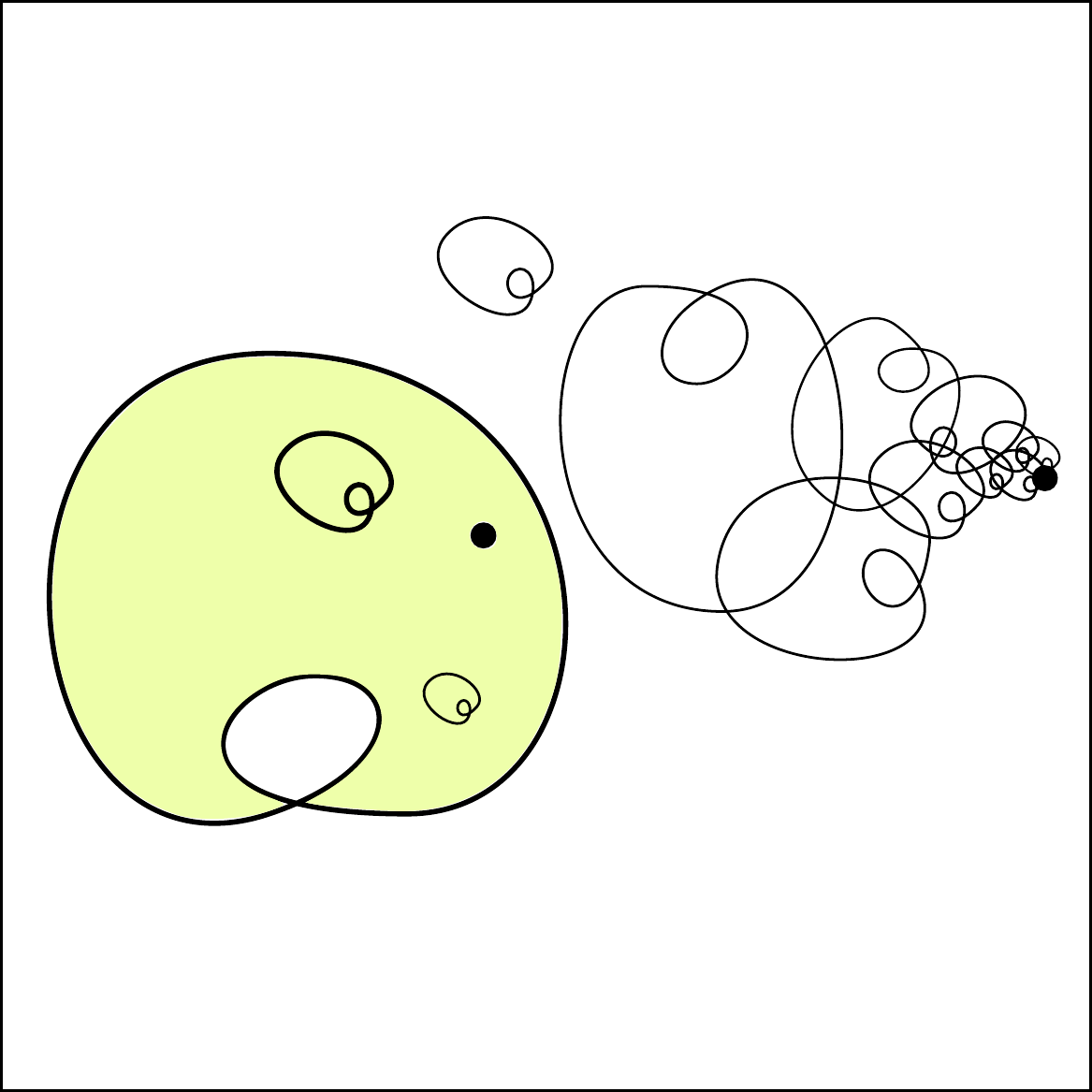} 
\setlength{\unitlength}{\textwidth}
\put(-0.85,0.015){(a)}
\put(-0.51,0.015){(b)}
\put(-0.17,0.015){(c)}
\put(-0.92,0.11){$p$}
\put(-0.92,0.18){\Large $\Omega$}
\put(-0.89,0.27){\large $\Gamma$}
\put(-0.41,0.125){$p$}
\put(-0.6,0.125){\Large $\Omega$}
\put(-0.5,0.29){\large $\Gamma$}
\put(-0.185,0.145){\scriptsize$p$}
\put(-0.28,0.14){\large $\Omega$}
\put(-0.17,0.26){\large $\Gamma$}
\put(-0.7,0.18){$\eta$}
\put(-0.44,0.2){$\zeta$}
\put(-0.02,0.15){$\eta$}
\caption{\small Sketch of possible setups in the proof of Theorem  \ref{invbasins}.  In case (b), $\zeta$ is an attracting fixed point. In cases (a) and (c),  $\eta$ is either an attracting fixed point or the point at infinity.}
\label{fig:sketch1}
\end{figure}
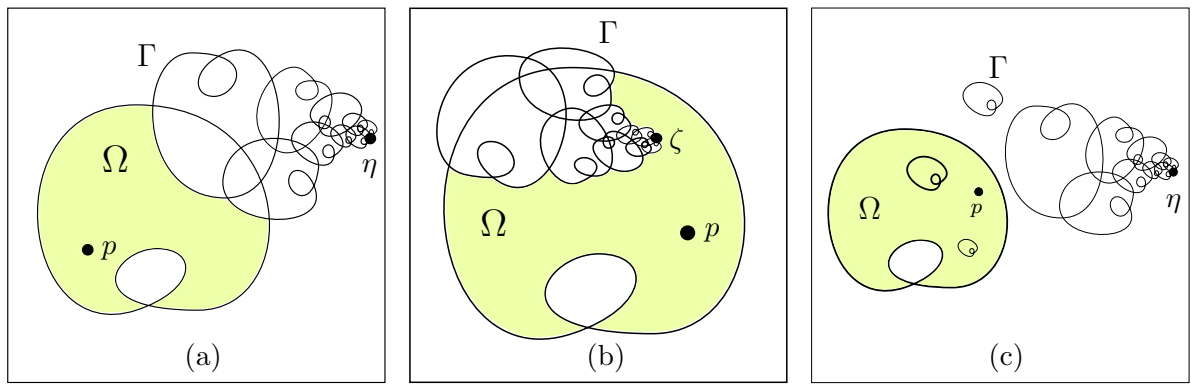

%%%%%%%%%%%%%%%%%
\subsection{Periodic Fatou components of period $p>1$, preperiodic components  and wandering domains}

\ 

Our goal in this section is to prove the following theorem. 
\begin{thm}[\bf (Pre)periodic and wandering Fatou components]\label{prop:periodic_wandering}
\label{thm:second}
Let $N$ be a Newton map and let $U$ be either a periodic Fatou component of minimal period $p>1$, a preperiodic component or a wandering domain. Then $U$ is simply connected. In particular $U$ cannot be a $p$-periodic Herman ring.  
\end{thm}

\begin{proof}
Assume that $U$ is multiply connected. By Lemma \ref{lem:poles-in-holes}, there
exist $n\geq 0$ and  a  simple closed curve $\gamma\subset N^n(U)$ surrounding a pole $p$ of $N$. Let $V$ be the Fatou component containing $N^n(U)$.  We denote by $\Omega$ the bounded connected component of $\mathbb C\setminus \gamma$ (i.e., $\gamma=\partial \Omega$).

Observe that $V$ cannot be invariant, since Theorem \ref{invbasins} ensures that invariant Fatou components are simply connected. Hence, $N(\gamma)\cap \gamma =\emptyset$ and there are three cases to be considered: 
\[
\Omega \subset \ext(N(\gamma)) \text{\ \ or\ \ }  N\left(\gamma \right) \subset \Omega \text{\ \ or\ \ }\gamma\subset
K\left(N\left(\gamma\right)\right).
\]

\begin{figure}[htp!] 
\centering
\includegraphics[width=0.32\textwidth]{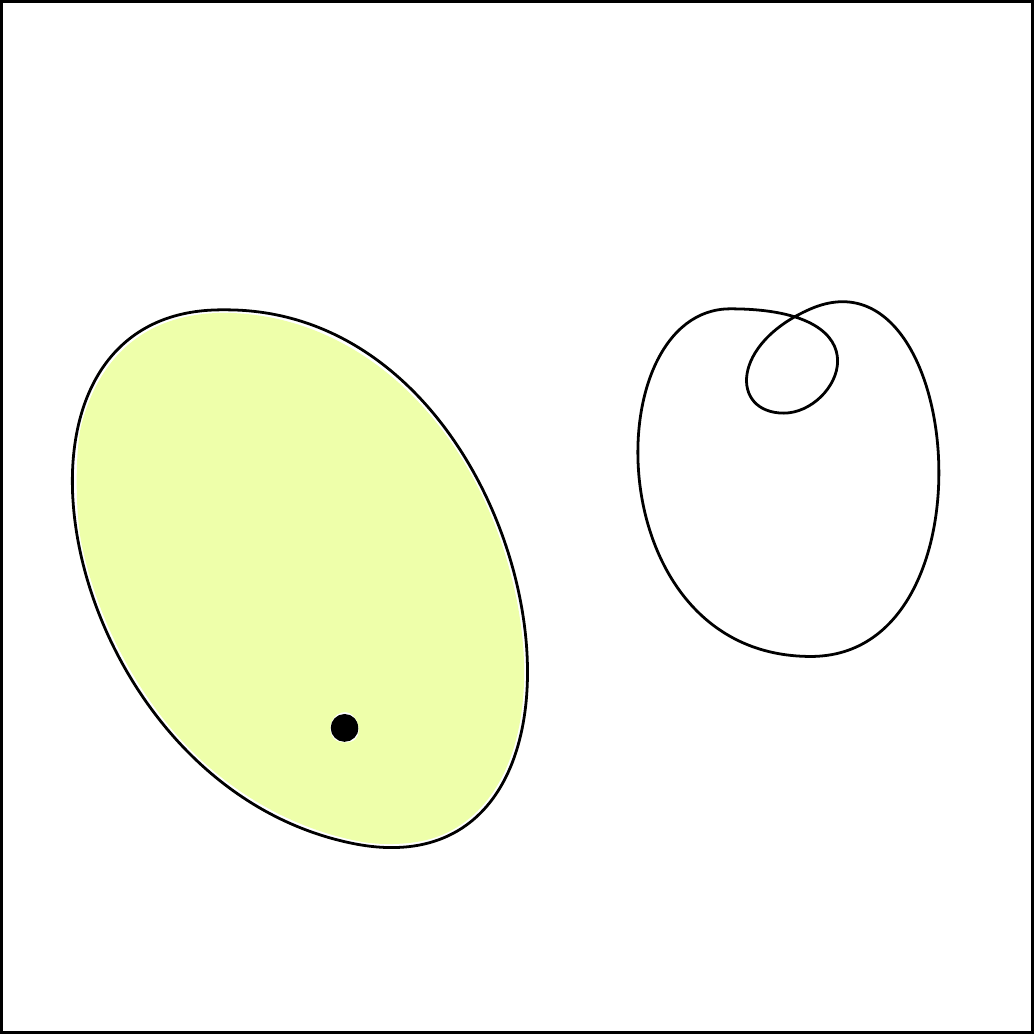} \hfill
\includegraphics[width=0.32\textwidth]{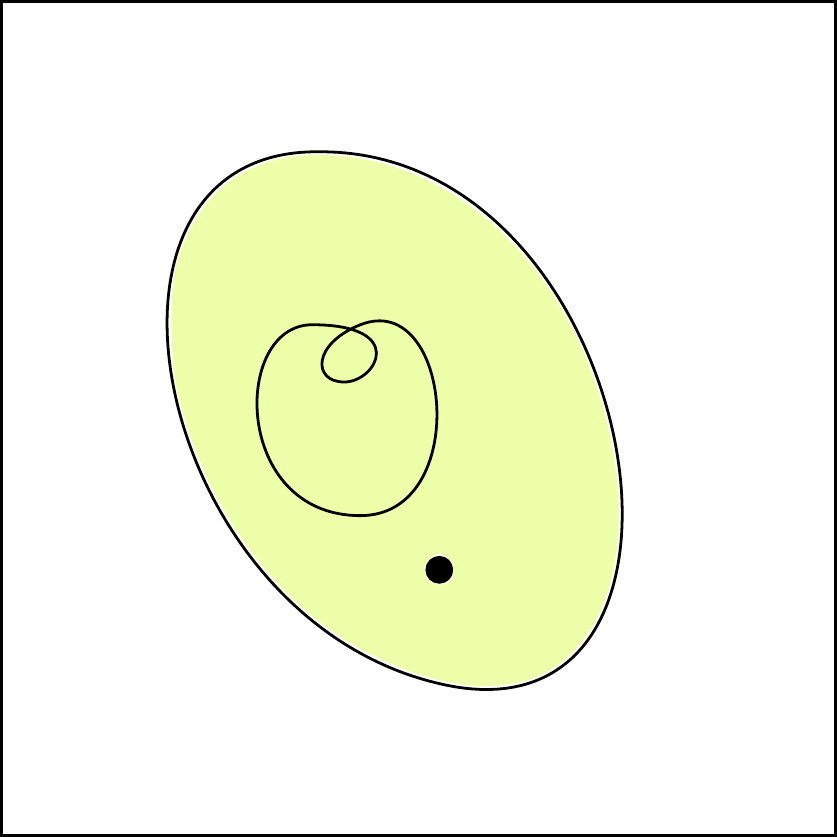} \hfill
\includegraphics[width=0.32\textwidth]{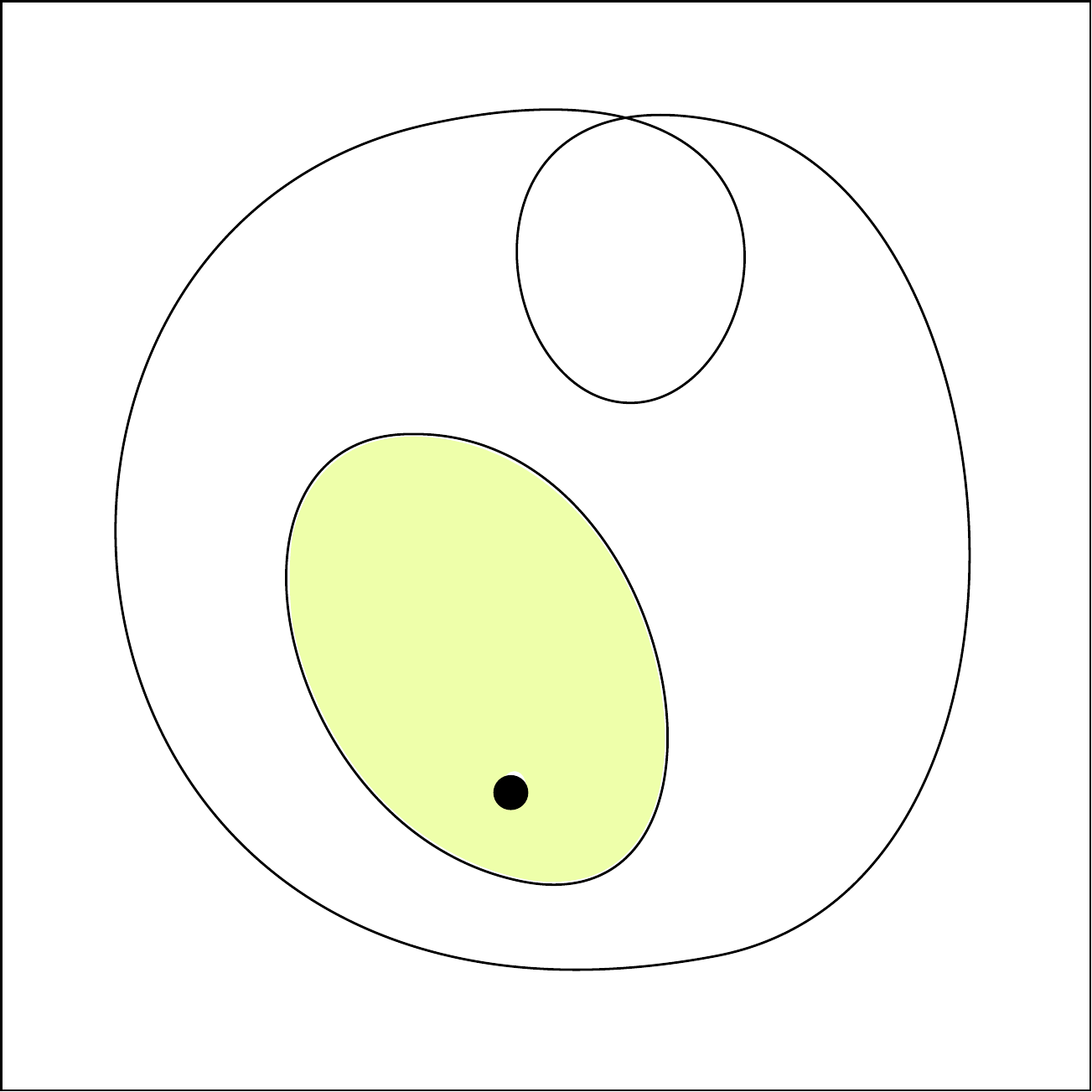} 
\setlength{\unitlength}{\textwidth}
\put(-0.85,0.015){(a)}
\put(-0.51,0.015){(b)}
\put(-0.17,0.015){(c)}
\put(-0.9,0.11){$p$}
\put(-0.96,0.185){\Large $\Omega$}
\put(-0.78,0.24){$N(\gamma)$}
\put(-0.9,0.23){$\gamma$}
\put(-0.5,0.082){$p$}
\put(-0.57,0.22){\Large $\Omega$}
\put(-0.48,0.25){$\gamma$}
\put(-0.484,0.15){\small $N(\gamma)$}
\put(-0.185,0.1){$p$}
\put(-0.13,0.15){$\gamma$}
\put(-0.22,0.15){\Large $\Omega$}
\put(-0.07,0.27){$N(\gamma)$}
\caption{\small Sketch of possible setups in the proof of Theorem  \ref{thm:second}.}
\label{fig:sketch2}
\end{figure}
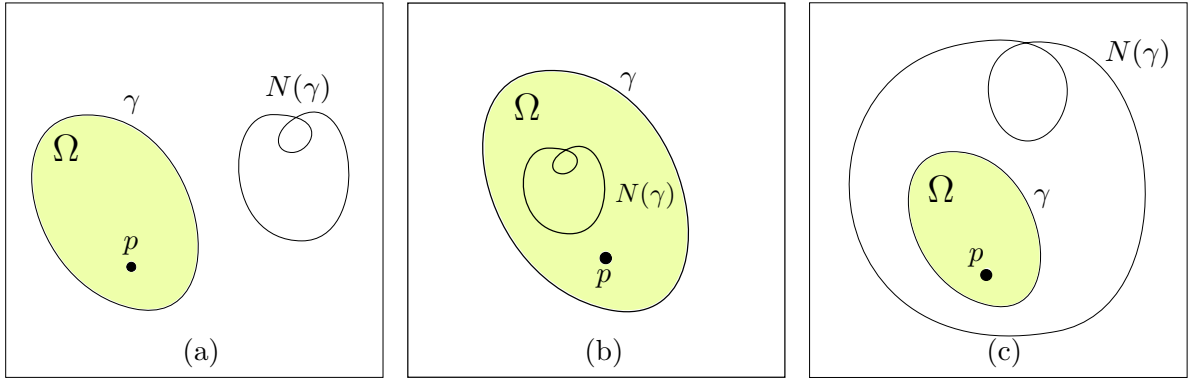

In the first case (see Figure \ref{fig:sketch2} (a)), Corollary \ref{cor:mapout1} with $X=\gamma$   implies that $N$ has a weakly repelling fixed point in $\Omega$, a contradiction. In the second case (see Figure \ref{fig:sketch2} (b)), Proposition \ref{prop:index} implies that $N$ has at least two fixed points in $\Omega$. Since $N$ is a  Newton map, these two fixed points in $\Omega$ are attracting. Their corresponding immediate basins are in the interior of $\gamma$ and hence they are bounded, which contradicts Proposition \ref{prop:unbounded_attracting_basins}. 

Now consider the remaining case (see Figure \ref{fig:sketch2} (c))
\[
\gamma\subset K\left(N\left(\gamma\right)\right).
\]
 
Notice that $N(\gamma) \subset N(V)$ also surrounds the pole $p$ and hence $N(V)$ is not forward invariant either. It follows that  $N^2(\gamma) \cap N(\gamma) =\emptyset$ 
%(but we cannot discard  that $N^2(\gamma) \cap \gamma \neq \emptyset$ since we might be in a periodic Fatou component of period $p=2$). 

There are two possible relative position between $N^2(\gamma)$ and $N(\gamma)$.
\[
 N^2(\gamma) \subset \ext(N(\gamma)) \text{\ \ or \ \ } 
 N^2(\gamma) \subset K(N(\gamma)).
 \]
 In the first case Corollary \ref{cor:mapout2} implies there is a weakly repelling fixed point in $K\left(N\left(\gamma\right)\right)$, a contradiction. In the second case, since $N^2(\gamma)$ and $N(\gamma)$ are disjoint, it follows that $N^2(\gamma)$ is contained in a component $\Omega'$ of $\C\setminus N(\gamma)$. 
 Notice that $\Omega'$ is bounded and simply connected. Since $\partial \Omega' \subset N(\gamma)$ and $N(\bd \Omega') \subset N^2(\gamma)$, we have that $N(\bd \Omega') \subset \Omega'$. We are then under the hypothesis of Proposition \ref{prop:index}, from which we conclude that $\Omega'$ contains a fixed point. This fixed point must be attracting and its basin is contained in $\Omega'$ because $\partial \Omega' \subset N(V)$ and $N(V)$ is not  invariant. But this is a contradiction since $N$ has no bounded attracting basins by Proposition  \ref{prop:unbounded_attracting_basins}.
\end{proof}

\bibliography{unified}

\end{document}